\documentclass[11pt]{article}

\usepackage{fullpage}
\usepackage{amsmath}
\usepackage{amsthm}
\usepackage{amssymb}
\usepackage{amsfonts}
\usepackage{subfig}
\usepackage{multirow}
\usepackage{verbatim}
\usepackage{enumerate}
\usepackage{array}

\newcommand{\todo}[2][yo]{}

\usepackage{tikz}
\tikzstyle{vertex}=[circle,draw=black,fill=black,inner sep=0mm,minimum size=1mm]
\tikzstyle{ar}=[<=1pt,>=stealth',semithick]

\usepackage{plevel}
\usepackage{TikzFigures}

\usepackage{graphicx}
\graphicspath{%
    {converted_graphics/}
    {/}
    {Images/}
}
\begin{document}

\title{Interval-Valued Rank in Finite Ordered Sets\thanks{  PNNL-SA-105144.}}
\date{}

\author{
	Cliff Joslyn\thanks{National Security Directorate, Pacific
Northwest National Laboratory. Corresponding author, cliff.joslyn@pnnl.gov,
206-552-0351.},
	Emilie Hogan\thanks{Fundamental Sciences Directorate, Pacific
Northwest National Laboratory}, and
	Alex Pogel\thanks{Physical Science Laboratory, New Mexico State University}
	}

\maketitle

\begin{abstract}

We consider the concept of rank as a measure of the vertical levels and
positions of elements of partially ordered sets (posets). We are motivated by the
need for algorithmic measures on large, real-world
hierarchically-structured data objects like the semantic hierarchies of
ontological databases. These rarely satisfy the strong property of gradedness, which
is required for traditional rank functions to exist. Representing such
semantic hierarchies as finite, bounded posets,
we recognize the duality of ordered structures to motivate rank functions
which respect verticality both from the bottom and from the
top. Our rank functions are thus interval-valued, and always exist,
even for non-graded posets, providing order homomorphisms to an interval
order on the interval-valued ranks. The concept of rank width arises
naturally, allowing us to identify the poset region with point-valued
width as its longest graded portion (which we call the ``spindle'').
A standard interval rank function is
naturally motivated both in terms of its extremality and on pragmatic
grounds. Its properties are examined, including the relationship to
traditional grading and rank functions, and methods to assess
comparisons of standard interval-valued ranks.

\end{abstract}

\tableofcontents

\section{Introduction}

A characteristic of partial orders and partially ordered sets (posets) as used in
order theory \cite{DaBPrH90} is that they are
amongst the simplest structures which can be called ``hierarchical'' in the sense of
admitting to descriptions in terms of {\em levels}. And yet we have found
that the mathematical development of the concepts of level, ``depth'', or
``rank'' is surprisingly incomplete.  Rank is used in order theory in two
related senses:

\begin{itemize}

\item First, the {\bf rank function} (if it exists) of a poset is a
    canonical scalar monotonic function representing the vertical level of
    each element. However, this concept has been literally one-sided,
    requiring a bias towards a particular ``pointing'' of the poset as
    being either ``top down'' or ``bottom up'', arbitrarily.

\item And secondly, a poset as a whole can be {\bf graded}, or ``have
    rank'', meaning that there exists such a rank function which assigns
    each element a {\em unique} vertical level consistent with its covering
    structure, such that a single step in the hierarchy results in a single
    increment or decrement of rank. But while graded posets dominate the
    primary results in formal lattice theory \cite{BiG40}, this is a strong
    property. Real-world data objects with hierarchical structure cannot
    generally be assumed to be graded, and in certain fields virtually none
    of them are.

\end{itemize}

We thus seek a concept of rank in general finite posets, whether
graded or not. We believe that this goal is satisfied naturally by extending
rank to be intervally-valued, and to use interval orders to compare
the vertical level of components of hierarchies.

Our motivation is from the kinds of hierarchically-structured data objects
which abound in computer science, and especially the semantic typing
systems of modern knowledge systems such as ontological databases and
object-oriented typing hierarchies. These structures are built around
{\em semantic hierarchies}
as their cores. These, in turn, are
collections of objects (classes or linguistic concepts) which are organized
into hierarchies such as taxonomic (subsumptive, ``is-a''); meronomic
(compositional, ``has-part''); and implication (logical, ``follows-from'')
relations.  Prominent examples include WordNet\footnote{\tt
wordnet.princeton.edu}  in the computational linguistics community
\cite{FeC98} and the Gene Ontology (GO\footnote{\tt www.geneontology.org}) in
the computational genomics community \cite{AsMBaC00}. \fig{consort_blank}
shows a portion of the GO. Black text indicates nodes representing biological
processes, and are arranged in a directed acyclic graph (DAG) representing
subsumption (e.g.\ ``DNA repair'' is a kind of ``DNA metabolism''). The names
of genes from three species of model organism are shown in colored text as
``annotations'' to these nodes, indicating that those genes perform those
functions.

\begin{figure}[htbp]
  \begin{center}
    \includegraphics[scale=0.33]{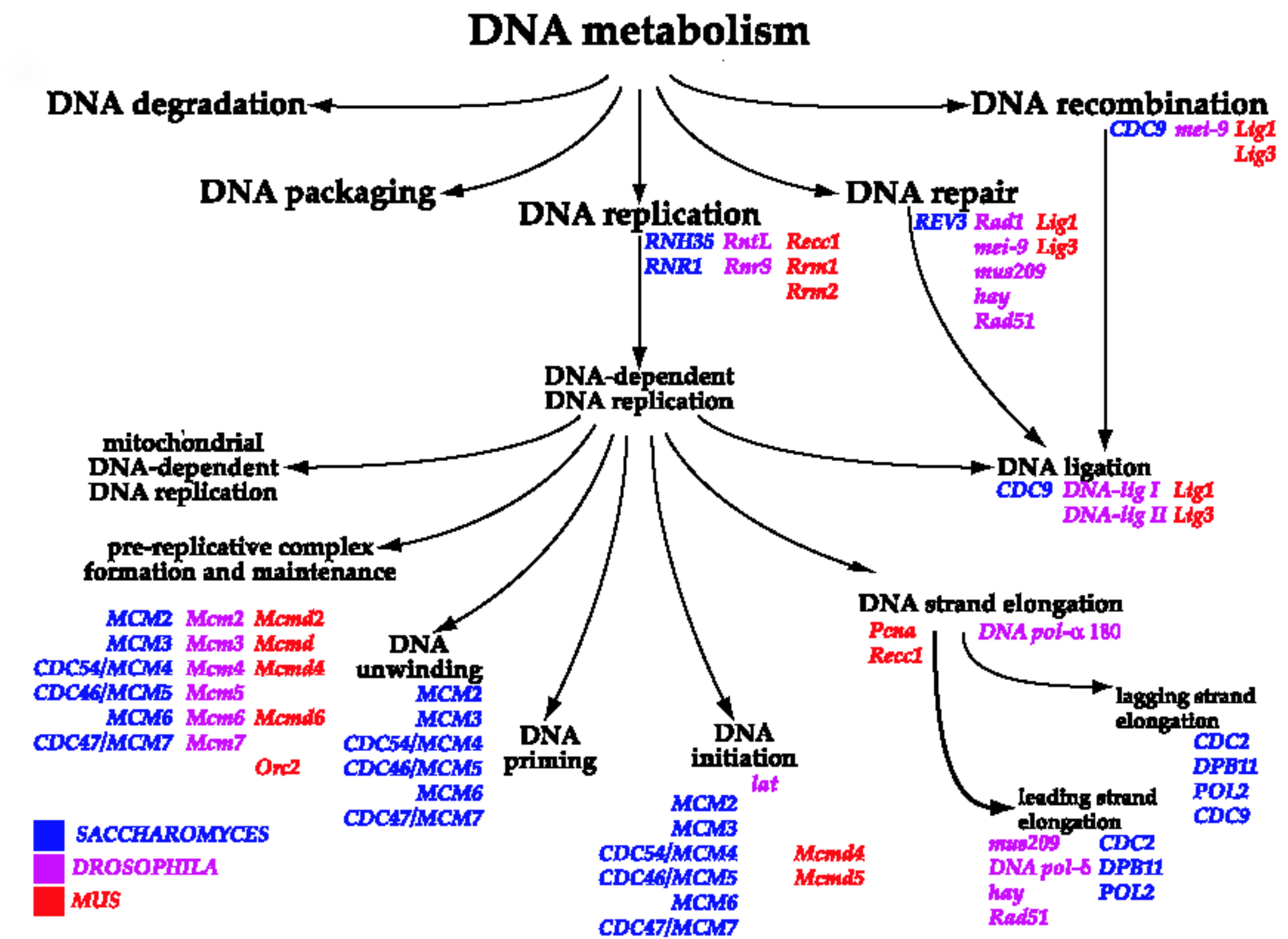}
  \end{center}
  \caption{A portion of the Biological Process branch of
the Gene Ontology (adapted from \cite{AsMBaC00}). The database is structured
as a large, top-rooted directed acyclic graph of genomic functional
categories, labeled with the genes of multiple species.}\label{consort_blank}
\end{figure}


\fig{consort_blank} shows only a small fragment of one portion of the GO. But
typical of many such real-world semantic hierarchies, the GO has grown to be
quite large, currently over thirty thousand concepts. As the number and size
of semantic hierarchies grows, it is becoming critical to have computer
systems which are appropriate for managing them, and especially important to
complement manual methods with algorithmic approaches to tasks such as
construction, alignment, annotation, and visualization. These tasks in turn
require us to have a solid mathematical grounding for the analysis of
semantic hierarchies, and especially the ability to measure entire
hierarchies and portions of hierarchies using such concepts as distances and
sizes of regions.

One particular aspect of hierarchical structure which needs to be measurable
is the placement of elements vertically relative to each other in ranks or
levels. This is critical, for example, in visualization applications.
\fig{cellcelllabels0} (from \cite{JoCMnS06}), shows a visualization of a GO
portion, wherein the vertical layout over a large hierarchical structure is
essential.

\myeps{.3}{cellcelllabels0}{An example ontology layout result, from
\cite{JoCMnS06}.}

As mathematical data objects, semantic hierarchies resemble
top-rooted trees. But with the presence of significant amounts of ``multiple
inheritance'' (nodes having more than one parent, as for ``DNA ligation'' in
\fig{consort_blank}), and also the possible inclusion of transitive links,
they must in general be represented at most as DAGs. And since the primary
semantic categories of subsumption and composition are transitive, the proper
mathematical grounding for such algorithms and measures is order theory,
representing semantic hierarchies as posets.

While order theory offers a fundamental formalization of hierarchy in
general, it has been largely neglected in knowledge systems technology, and
we believe it has promise in the development of tools for managing semantic
hierarchies. The most prominent approaches to managing the GO and WordNet are
effectively {\it ad hoc} and {\it de novo} (e.g.\ \cite{BuAHiG06}). And
conversely, even outside of the domain of knowledge systems applications, the
precise formulation in ordered sets of concepts such as distance and
dispersion are little known and in many cases underdeveloped in the first
place. Our prior work \cite{JoC04a,JoCHoE10,JoCMnS04b,KaTJoC06,VeKCoJ05e} has
advanced the use of measures of distance and similarity, and
characterizations of mappings and linkages within semantic hierarchies.

Semantic hierarchies are rarely graded, so
the
standard concepts of rank and grade are not especially useful for dealing
with measuring and laying out their vertical levels.
In this paper, we seek a coherent formulation of
vertical element placement in posets. We begin with some notational elements
about posets, maximal chains, maximum chains (which we call the {\em spindle}), and
order morphisms. We then introduce concepts and notation for integer
intervals and operations thereon, including our treatment
of interval orders, including strong, weak, and subset orders.
We use these tools to consider the general sense of rank
and vertical rank in posets as a necessarily two-sided concept, motivating
rank as an interval-valued function inducing a strict order morphism to some
interval order.  Cases are considered, and then a strict version of an
interval rank function is introduced, which preserves strict homomorphism not
just for the intervals overall, but also for their constituent endpoints as
well.

We can identify a particular interval rank function for the weak
interval order, which we call the standard interval rank function. We point
out that this standard interval rank function is special both on the basis of
its extremality, and as derived from a procedural approach to natural
scalar-valued rank. We consider some properties of the standard
interval rank, and the use of interval analysis to be able to
compare interval-valued
ranks in terms of their interval-valued separation.
In this way, we can gain a sense of the vertical distances between
even non-comparable poset elements.


\section{Preliminaries}
Throughout this paper we will use $\N$ to denote  $\{ 0, 1, \ldots \}$, the
set of integers greater than or equal to 0, and for $N \in \N$, the set of
integers between 0 and $N$ will be denoted $\N_N \define \{ 0, 1, \ldots, N
\}$.

\subsection{Ordered Sets}\label{OrderedSetsDefns}


See e.g.\ \cite{DaBPrH90,ScB03,TrW92} for the basics of order theory, the
following is primarily for notational purposes.

Let $P$ be a finite set of elements with $|P| \ge 2$, and $\le$ be a binary
relation on $P$ (a subset of $P^2$) which is reflexive, transitive, and
antisymmetric. Then $\le$ is a {\bf partial order}, and the structure $\poset
= \tup{ P, \le }$ is a {\bf partially ordered set} (poset). $a < b$
means that $a \le b$ and $a \neq b$, and is a {\bf
strict order} on $P$, which is an irreflexive partial order. A strict
order $<$ can be turned into a partial order through {\bf reflexive closure}:
$\le \define < \un \,\{ \tup{ a, a } \in P^2 \}$.

For any pair of elements $a,b \in P$, we say $a \le b \in P$ to mean that $a,b
\in P$ and $a \le b$. And for $a \in P, Q \sub P$, we say $a \le Q$ to mean
that $\forall b \in Q, a \le b$. If $a \le b \in P$ or $b \le a \in P$ then
we say that $a$ and $b$ are {\bf comparable}, denoted $a \sim b$. If not,
then they are {\bf incomparable}, denoted $a \| b$. For $a,b \in P$, let $a
\cover b$ be the {\bf covering relation} where $a \le b$ and $\nexists c \in
P$ with $a < c < b$.

A set of elements $C \sub P$ is a {\bf chain} if $\forall a, b \in C, a \sim b$.
If $P$ is a chain, then $\poset$ is called a {\bf total order}. A chain $C
\sub P$ is {\bf maximal} if there is no other chain $C' \sub P$ with $C \sub
C'$. Naturally all maximal chains are {\bf saturated}, meaning that $C = \{
a_i \}_{i=1}^{|C|} \sub P$ can be sorted by $\le$ and written as $C = a_1
\cover a_2 \cover \ldots \cover a_{|C|}$. The {\bf height} $\height(\poset)$
of a poset is the size of its largest chain. Below we will use $\height$
alone for $\height(\poset)$ when clear from context.

For any subset of elements $Q \sub P$, let $\psub{Q} = \tup{ Q, \le_Q }$ be the
sub-poset determined by $Q$, so that for $a,b \in Q$, $a \le_{Q} b \in Q$ if
$a \le b \in P$.

For any element $a \in P$, define the {\bf up-set} or {\bf principal filter}
$\up a \define \{ b \in P \st b \ge a \}$, {\bf down-set} or {\bf principal
ideal} $\down a \define \{ b \in P \st b \le a \}$, and {\bf hourglass}
$\Xi(a) \define \up a \un \down a$. For $a \le b \in P$, define the {\bf
interval} $[a,b] \define \{ c \in P \st a \le c \le b \} = \up a \int \down b
$.

For any subset of elements $Q \sub P$, define its maximal and minimal elements
as
	\[ \Maxxx(Q) \define
		\{ a \in Q \st {\nexists} b \in Q, a < b \}	\sub Q  \]
	\[ \Minnn(Q) \define
		\{ a \in Q \st {\nexists} b \in Q, b < a \} \sub Q,	\]
	called the {\bf roots} and {\bf leaves} respectively. Except where
noted, in this paper we will assume that our posets $\poset$ are bounded, so
that $\bot \le \top \in P$ with $\Maxxx(P) = \{ \top \}, \Minnn(P) = \{ \bot
\}$. Since we've disallowed the degenerate case of $|P|=1$, we have $\bot <
\top \in P$. All intervals $[a,b]$ are bounded sub-posets, and since $\poset$
is bounded, $\forall a \in P, \up a = [ a, \top ], \down a = [ \bot, a ]$,
and thus $\bot,\top \in \Xi(a) \sub \poset = [\bot,\top]$. An example of a
bounded poset and a sub-poset expressed as an hourglass is shown in
\fig{allex2}.

\mytikz{\bddExWithHourglass}{>=latex, line width=0.75pt}{allex2}{(Left) The
Hasse diagram (canonical visual representation of the covering relation
$\cover$) of an example bounded poset $\poset$. (Right) The Hasse diagram of
the sub-poset $\psub{\Xi(J)}$ for the hourglass $\Xi(J) = \up J \un \down J =
[ J, \top ] \un [ \bot, J ] = \{ J,C,K,\top \} \un \{\bot,J\} = \{ \bot, C,
J, K, \top \}$.}

The following properties are prominent in lattice theory, and are available
for some of the highly regular lattices and posets that appear there (see
e.g.\ Aigner \cite{AiM79}).

\begin{defn}[(Jordan-Dedekind Condition)]
A poset $\poset$ has the Jordan-Dedekind condition (we say that $\poset$ is
JD) when all saturated chains (note: not just maximal chains) connecting two
elements $a \le b \in P$ are the same size.
\end{defn}

\begin{defn}[(Rank Function and Graded Posets)]
For a top-bounded poset $\poset$, a function
$\func{\rho}{P}{\N_{\height-1}}$ is a {\bf rank function} when $\rho(\top) =
0$ and $\forall a \cover b \in P, \rho(a) = \rho(b) - 1$.\footnote{Aigner's
posets are actually ``pointed'' the other way, defined on bottom-bounded
posets, so his definition is actually with $+$ here. It is otherwise
identical. The value of our alternate pointing will be discussed in
\sec{sec-intrank} below, and will also arise in Definition \ref{posetrank}.}
A poset $\poset$ is {\bf graded}, or {\bf fully graded}, if it has a rank
function.
\end{defn}

\begin{thm}\cite{BiG40} \label{graded}
A poset $\poset$ is graded iff it is JD.
\end{thm}

Let $\chains(\poset) \sub \pow{P}$ be the set of all maximal chains of
$\poset$. We assume that $\poset$ is bounded, therefore $\forall C \in
\chains(\poset), \bot,\top \in C$. For any element $a \in P$, define its {\bf
centrality} as the length of the largest maximal chain it sits on, i.e., the
height of its hourglass
	\[ S(a) \define
		\height( \Xi(a) ) = \height( \up a ) + \height( \down a ) - 1. \]
	We will refer to the {\bf spindle chains} of a poset $\poset$ as
the set of its maximum length chains
	\[ \spindle(\poset) \define
		\left\{ C \in \class( \poset ) \st |C| =
			\height \right\}.	\]
	The {\bf spindle set}
	\[ I(\poset) \define \Un_{C \in \spindle(\poset)} C \sub P	\]
	is then the set of {\bf spindle elements}, including any elements which sit on
a spindle chain. Note that if $P$ is nonempty (as we require)
then there is always at least
one spindle chain and thus at least one spindle element, so also $\spindle(\poset),
I(\poset) \neq \emptyset$.

In our example in \fig{allex2}, we have $\height =
5$, $|\chains(\poset)| = 6$, $\spindle(\poset) = \{ \bot \cover A \cover H
\cover K \cover \top \}$, and $S(J) = \height( \up J ) + \height(\down J ) -
1 = 4$.

Given two posets $\poset_1 = \tup{ P, \le_P }$ and $\poset_2 = \tup{ Q, \le_Q
}$, a function $\func{f}{P}{Q}$ is an \textbf{order homomorphism}
provided $a \le_P b \in P \implies f(a) \le_Q f(b)$. Compare to the stronger, and possibly more familiar
notion of an \textbf{order embedding} \cite{DaBPrH90} where $a \le_P b \iff f(a) \le_Q f(b)$.
In this paper we will only be using order homomorphisms which simply map
the structure of one partial order into another without the reverse needing to be true.
We can also say that $f$
\textbf{preserves the order} $\le_P$ into $\le_Q$, and is an \textbf{isotone}
mapping from $\poset_1$ to $\poset_2$. If $\forall a <_P b \in P, f(a) <_Q
f(b) \in Q$ then we say that $f$ does all this \textbf{strictly}. If instead
$f$ is an order homomorphism from $\poset_1$ to the dual $\tup{ Q, \ge_Q }$,
then we say that $f$ \textbf{reverses the order}, or $f$ is an
\textbf{antitone} mapping. If $\func{f}{P}{Q}$ is an order homomorphism, then
we can denote $f(\poset_1) \define \tup{ f(P), \le_f }$ as the
\textbf{homomorphic image} of $\poset_1$, with
	\[ f(P) \define \{ f(a) \st a \in P \} \sub Q,	\quad
		\le_f\, \, \define
            \leq_Q\!\left|_{f(P)\times f(P)}\right. \]
When clear from context, we will simply re-use $\le$ as the relevant order to
its base set, e.g.\ for an isotone $\func{f}{P}{Q}, a \le b \in P \implies
f(a) \le f(b) \in Q$.


Finally, we recognize $\tup{ \N, \le }$ as a total order using the normal
numeric order $\le$, and observe that for any bounded poset $\poset$, the
functions $\height(\up \cdot), \height(\down \cdot): P \rightarrow
\N_{\height}$ induce strict antitone and isotone order morphisms,
respectively. That is, if $a \cover b$ then
\[ \height(\up a) \geq \height(\up b)+1 > \height(\up b) \qquad\text{and}\qquad
    \height(\down a) < \height(\down a)+1 \leq \height(\down b). \]

\subsection{Numeric Intervals: Operations and Orders} \label{intervals}

Since our rank functions will be interval-valued, we explicate the concepts
surrounding the possible ordering relations among intervals. Our formulation
is a bit nonstandard, but helpful for this particular application. Our view
actually resonates with algebraic approaches to interval analysis used in
artificial intelligence, such as Allen's interval algebra
\cite{AlJ83,LaPMaR87,LiG90}. But these methods are not order-theoretical, nor
are we aware of prior use of our weak interval order $\le_W$ in interval
analysis proper. We begin by defining operations on intervals before going
into orders among intervals.

$\tup{\R,\le}$ is a total order, so for any $x_* \le x^* \in \R$, we can
denote the (real) interval $\bar{x} = [ x_*, x^* ]$, and $\overline{\R}$ as
the set of all real intervals on $\R$, so that $\bar{x} \in \overline{\R}$.
Additionally, for $N \in \N$ let $\overline{N}$ be the set of all intervals
whose endpoints are nonnegative integers $\le N$. We define the {\bf interval
midpoint} $\hat{\bar{x}} = \frac{x_* + x^*}{2} \in \R$ and {\bf interval
width} $W(\bx) \define | x^* - x_* | \in \R$.

Operations on real intervals (see e.g.\ \cite{MoR79}) are defined setwise, so
that for $\bar{x},\bar{y} \in \overline{\R}$ and a unary or binary operation
$\bullet$,
	\[\bullet(\bar{x}) \define
		\{ \bullet(z): z \in \bar{x} \}, \qquad \bar{x} \bullet \bar{y} \define\{ z_1 \bullet z_2: z_1 \in \bar{x}, z_2 \in \bar{y} \}.\]
	This yields specific real interval operations of {\bf addition}
$\bx + \by \define [ x_* + y_*, x^* + y^* ]$ (often referred to as the
Minkowski sum \cite{BeR1966}); {\bf subtraction} $\bx - \by \define [ x_* -
y^*, x^* - y_* ]$; and {\bf absolute value} $|\bx| = [ |\bx|_*, |\bx|^* ]$,
where
\begin{align*}
  |\bx|_* &\define \left\{\begin{array}{ll}
                            0 & x_*x^* \leq 0\\
                            \minn( |x_*|, |x^*| ) & x_*x^* > 0
                          \end{array}\right.\\
  |\bx|^* &\define \maxx( |x_*|, |x^*| ).
\end{align*}
We also have the {\bf separation} $\| \bx, \by \| \define | \bx - \by|$, an
interval valued function of two intervals which represents the interval
between the minimum and maximum difference between any two chosen points, one
from each interval.

In the literature \cite{BaKFiP72,FiP85,FiP1985book,TrW97} the standard
ordering on a set of intervals $P$ is a binary relation $\le$ satisfying the
Ferrers property \cite{BuA03,DiSDeB11}, so that $\forall x,y,z,w \in P, x \le
y, z \le w$ implies $x \le w$ or $z \le y$. This results in the almost
completely universal recognition of the term ``interval order'' to mean what
we call here the {\bf strong interval order} $\le_S$. While it is the only
interval order we consider here which satisfies Ferrers, it is only one
possible order on real intervals. It is the strongest, but also the least
useful for our purposes compared to the others, although they are much less
widely recognized.

	\begin{defn}[(Strong Interval Order)] \label{strong}
	Let $<_S$ be a strict order on $\overline{\R}$ where $\bar{x} <_S
\bar{y}$ iff $x^* < y_*$. Let $\le_S$ be the reflexive closure of $<_S$, so
that $\bar{x} \leq_S \bar{y}$ iff $x^* < y_*$ or $\bar{x} = \bar{y}$.
	\end{defn}

In addition to the strong interval order $\le_S$, some \cite{TaP96b}
recognize  $\sub$ as an {\bf interval-containment} or {\bf subset interval
order}.

	\begin{defn}[(Subset Interval Order)]
	Let $\sub$ be a partial order on $\overline{\R}$ where $\bar{x}
\sub \bar{y}$ iff $x_* \ge y_*$ and $x^* \le y^*$.
	\end{defn}

But one of the most natural ordering relations between two intervals is given
by the product order $\le \times \le$ based on $\tup{ \R, \le }$. We call
this the {\bf weak interval order} (not to be confused with Fishburn's weak
order \cite{FiP1985book}, which is different).

	\begin{defn}[(Weak Interval Order)] \label{weak}
	Let $\le_W$ be a partial order on $\overline{\R}$ where $\bar{x}
\le_W \bar{y}$ iff $x_* \le y_*$ and $x^* \le y^*$.
	\end{defn}

These three orders on $\overline{R}$ are related as follows.

\begin{itemize}

\item As their names suggest, the strong order implies the weak order: if
    $\bar{x} \le_S \bar{y}$ then $\bar{x} \le_W \bar{y}$.
\item The weak order and subset order are related. Notice that if two
    intervals are properly weakly ordered (i.e., $\bar{x} <_W \bar{y}$ and
    none of the endpoints are equal) then the intervals are \emph{not}
    comparable in the subset order. Dually, if two intervals are properly
    subset ordered ($\bar{x} \subset \bar{y}$ and none of the endpoints are
    equal) then the intervals are \emph{not} comparable in the weak order.
    This notion of being comparable in one or the other, but not both, is
    called \textbf{conjugacy}. It is an interesting topic, but one we leave
    for consideration in future work \cite{JoCHoE14a}.


\end{itemize}

Based on these observations, we can identify two additional orders available
on real intervals:

\begin{itemize}

\item The dual to $\le_W$ is $\bar{x} \ge_W \bar{y}$, so that $\ge_W = \ge
    \times \ge$.

\item The dual to the subset order is the superset order, so that
    $\supseteq = \le \times \ge$.

\end{itemize}

%

Here we introduce the notation $\bar{x} \circ_\le \bar{y}$ to mean ``properly
intersecting from the left'', meaning
	\[ \bar{x} \circ_\le \bar{y} \define
		\bx \le_W \by, \bx \not\le_S \by.	\]
	There is also the dual ``properly intersecting from the right'':
	\[ \bx \circ_\ge \by \define
		\bx \ge_W \by, \bx \not\ge_S \by.	\] We will refer back to this in
Section  \ref{quant}. Note that $\circ_\le$ and $\circ_\ge$ are not
themselves partial orders, since they are not transitive. But when we assume
for simplicity that no two of $x_*,x^*,y_*$, and $y^*$ are equal (per
discussion above), then any pair of intervals $\bx,\by$ will stand in exactly one
of these three binary relations $\bx \bowtie \by$ for $\bowtie \in \{ <_S,
\circ_\le, \subset \}$, or their duals.

Later we will be interested in comparing intervals standing in different
order relations $\bowtie$; that is, calculating differences $\by-\bx$ and
separations $\|\bx,\by\|$. Let $\bar{\alpha} = \bar{y}-\bar{x} = [ y_* - x^*,
y^*-x_*]$, and consider the three situations shown in \fig{intorders1}, again
assuming that no two of $x_*,x^*,y_*$, or $y^*$ are equal. Then for the
three cases $\bowtie \in \{ <_S, \circ_\le, \subset \}$, the quantitative
relations among the components of $\bx,\by,\bar{a}=\by-\bx,\|\bx,\by\|$, and
the width $W(\|\bx,\by\|)$ are shown in Table \ref{taborders}. Similar
results hold for the dual cases $\bowtie \in \{ >_S, \circ_\ge, \supset \}$.

\begin{figure}
  \begin{center}
    \includegraphics[scale=0.4]{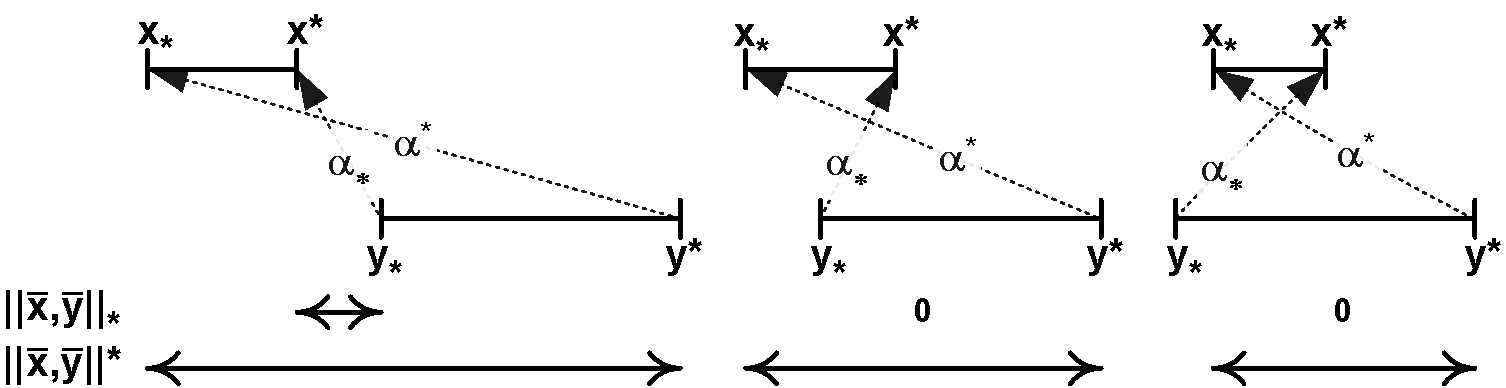}
  \end{center}
  \caption{The three qualitative interval relations,
including the separation interval $\|\bx,\by\| = [ \|\bx,\by\|_*,
\|\bx,\by\|^* ]$. (Left) $\bx <_S
\by$. (Center) $\bx \circ_\le \by$.  (Right) $\bx \subset \by$.}\label{intorders1}
\end{figure}


\begin{table}

\begin{center}

\begin{tabular}{>{$}c<{$}|>{$}l<{$}||*{3}{>{$}c<{$}|}>{$}c<{$}}

\multicolumn{2}{c||}{$\bx \bowtie \by$}	 &
	\multicolumn{2}{c|}{$\bar{\alpha} = \by - \bx=[ y_*
- x^*, y^*-x_*]$} &
	\|\bx,\by\| & W(\|\bx,\by\|) \\

\hline

\bx <_S \by &	x_* < x^* < y_* < y^* &
	[0,0] <_S \bar{\alpha} &
	0 < \alpha_* < \alpha^* & \bar{\alpha} = [\alpha_*,\alpha^*] &
	\alpha^* - \alpha_* \\

\bx \circ_\le \by & x_* < y_* < x^* < y^* &
	[0,0] \subset \bar{\alpha} &
	\alpha_* < 0 < -\alpha_* < \alpha^* & [0,\alpha^*] &
	\alpha^*	\\

\bx \subset \by &	y_* < x_* < x^* < y^* &
	[0,0] \subset \bar{\alpha} &
	\alpha_* < 0 < \alpha^* & [0,\Maxxx(-\alpha_*,\alpha^*)] &
	\Maxxx(-\alpha_*,\alpha^*) \\

\end{tabular}

\caption{Relations among quantities of interval difference depending on order
relation, assume in all cases that no two of $x_*,x^*,y_*$, and $y^*$ are
equal.}

\label{taborders}

\end{center}

\end{table}

In closing this section, we note that while these ordering relations
$\le_S,\sub$, and $\le_W$ are defined here for real intervals in the chain
$\tup{\R,\le}$, in fact they are also available for general intervals in
arbitrary posets $\tup{P,\le}$, something which we have started to explore
elsewhere \cite{ZaFKrV12}.

\section{Interval Rank} \label{sec-intrank}

Consider again the fragment of the GO in \fig{consort_blank}. This is typical
of our problem domain, where data structures are top-bounded DAGs, with a
moderate amount of multiple inheritance (multiple parents per node),
branching downward very strongly (on average many children per
node, and many more than parents), and whose transitive closures are not
required to be join semi-lattices, but typically are (in other words, it may
be the case that $\exists a,b \in P, |\Minnn( \up a \int \up b )| > 1$). We
model these structures as finite bounded posets $\poset$ by taking their
transitive closures and augmenting $P$ to include a bottom bound $\bot \in P$
such that $\forall a \in P, \bot \le a$. We also include a top bound if one
doesn't already exist.

For pragmatic purposes related to historical usage in the computer science
community, we want $\poset$ to be ``pointed'' so that the top bound $\top \in
P$ is up, but with rank 0, with ranks growing as we descend and shrinking as
we ascend. This introduces regrettable terminological and notational
complexity, and in particular the need to work in origin zero, which makes
counting difficult; and primarily with antitone (order reversing), rather
than isotone (order preserving), functions, as will be seen below.

Our development then begins in earnest by considering the typical
approach used in applications: count the vertical level of a node as how ``far'' it is
down from the top $\top \in P$, usually in terms of some kind of path
length, and perhaps relative to the total height
$\height$. In \fig{consort_blank}, $\top=$ ``DNA metabolism'', and
$\height=5$ (here considering $\po$ unbounded). However, compare the node
``DNA degradation'' with the node ``lagging strand elongation''. Despite
being in quite different apparent vertical locations relative to $\top$ (one
and four down from $\top$ respectively), they share something in common,
namely that they are both {\em leaves}. Thus in fact they (and all the
leaves) are {\em at the bottom}, no matter {\em how} far they may {\em also}
be from the top. So once the bottom $\bot$ is inserted, these leaves $L = \Minnn(P)$
of our original DAG are, in fact, all also ``one up'' from it: $\bot \cover
L$.

So it is imperative to consider the level of an element $a \in P$ not as a
one-sided distance  from a privileged direction in which the poset has been
pointed, but as a {\em joint} concept, involving the distance of $a$ from the
top $\top$, yes, but also from the bottom $\bot$, all in the context of the
total height of the poset. And furthermore, these distances are separate,
since as we have seen, elements can be either close to or far from either the
top or bottom independently. This motivates considering overall rank as being
better represented by these two independent numerical ``levels'', and thus as
a real interval between them.

\subsection{Interval Rank Functions}

We use these principles to characterize, with as few preconsiderations
as possible, the vertical structure of a poset in terms of a rank function
value $R(a)$ for each poset element $a \in P$. We favor rank functions
satisfying:

\begin{itemize}
  \item Rank $R$ should be an interval-valued function on $P$.
  \item Endpoints of these interval-valued ranks should be integers between
      0 and $\height-1$ (to accommodate our origin zero counting).
  \item Ranks should be monotone, and in particular preferably antitone
      (increasing as we descend from $\top \in P$, and decreasing as we
      ascend from $\bot \in P$) rather than isotone (the reverse).
	\end{itemize}

We begin formalizing this as follows.

\begin{defn}[(Interval Rank Function)]\label{IntRank}
Let $\poset = \tup{P,\le}$ be a poset and $\sqsubseteq$ an order on real
intervals.  Then a function, $\func{R_\sqsubseteq}{P}{\overline{\N}}$, with
$R_\sqsubseteq(a) = [ r_*(a), r^*(a) ]$ for $a \in P$, is an {\bf interval
rank function for $\sqsubseteq$} if $R_\sqsubseteq$ is a strict order
homomorphism $\po \mapsto \tup{\overline{\N},\sqsubseteq}$, i.e., for all $a
< b$ in $\po$ we have $R_\sqsubseteq(a) \sqsubset R_\sqsubseteq(b)$. Let
${\bf R}_\sqsubseteq(\poset)$ be the set of all interval rank functions
$R_\sqsubseteq$ for the interval order $\sqsubseteq$ on a poset $\poset$.
\end{defn}


Below we will frequently omit the subscript and simply use $R$ when clear
from context. Note the value provided by a strict order homomorphism in Def.\
\ref{IntRank}, since without strictness any function $R(a)$ which $\forall a \in P,
R(a)=[r,r]$ for some constant $r \in \N$ would be a valid interval rank
function for any interval order.

When $r_*(a) = r^*(a)$, then denote $r(a) \define r_*(a) = r^*(a)$. Let
$W_R(a) \define W(R(a)) \in \N_N$ be the {\bf rank width} for interval rank
function $R$ (or just $W(a)$ when clear from context), and
	\[ \widehat{R(a)} \define \frac{ r_*(a) + r^*(a) }{2} \in \R	\]
	be the {\bf rank midpoint}. We will say that an element $a \in P$ is {\bf
precisely ranked} if $W(a) = 0$, and a set of elements $Q \sub P$ is precisely
ranked if all $a \in Q$ are precisely ranked.

Where Def. \ref{IntRank} captures monotonicity of the interval ranks, does it
follow that their constituent endpoints are also monotonic? The answer is yes
for the interval orders defined in Section \ref{intervals}.

\begin{prop} \label{EndpointMonotone} 
Let $R_\sqsubseteq(a) = [r_*(a),r^*(a)]:\po \rightarrow \ints$ be an interval
rank function for interval order $\sqsubseteq$. Then, $r_*(a)$ and $r^*(a)$
are monotonic functions iff $\sqsubseteq \in \{ \le_W, \ge_W, \subseteq,
\supseteq\}$. In particular:
\begin{enumerate}[(i)]
  \item $r_*$ and $r^*$ are isotone iff $\sqsubseteq = \leq_W$ .
  \item $r_*$ and $r^*$ are antitone iff $\sqsubseteq = \geq_W$.
  \item $r_*$ is antitone and $r^*$ is isotone iff $\sqsubseteq =
      \subseteq$.
  \item $r_*$ is isotone and $r^*$ is antitone iff $\sqsubseteq =
      \supseteq$.
\end{enumerate}
\end{prop}

\begin{proof}\mbox{}
\begin{description}
  \item[$(\Longleftarrow)$] Here we assume that $\sqsubseteq \in
      \{\le_W,\ge_W,\subseteq,\supseteq\}$ and need to show that $r_*$ and
      $r^*$ are monotonic functions of the forms above. We will just prove
      (i) here and observe that the rest of the cases follow analogously.
      Let $R_{\leq_W}(a) = [r_*(a),r^*(a)]$ be an interval rank function
      for the weak interval order. Then, if $a<b \in P$ we know from Def.
      \ref{IntRank} that $R_{\leq_W}(a) <_W R_{\leq_W}(b)$. But this simply
      means that
      \[ r_*(a) \leq r_*(b) \text{ and } r^*(a) \leq r^*(b), \]
      and therefore both $r_*$ and $r^*$ are isotone.
      Similarly for the other three cases we see that requiring a strict
      order homomorphism implies monotonicity conditions on the endpoints of
      the intervals.

  \item[$(\Longrightarrow)$] Now we assume that $r_*$ and $r^*$ are
      monotonic functions, and that $R_\sqsubseteq=[r_*,r^*]$ is an
      interval rank function. We must show that each of the four ways in
      which $r_*$ and $r^*$ are both monotonic we have a unique interval
      order. This time, let us assume that $r_*$ is antitone and $r^*$ is
      isotone (so we are in case (iii)). If $a < b \in \po$ we know that
      \[ r_*(a) \geq r_*(b) \quad\text{ and }\quad r^*(a) \leq r^*(b).\]
      This is exactly $R_\sqsubseteq(a) \subseteq R_\sqsubseteq(b)$.
	But of course we assumed $R_\sqsubseteq$ was an interval rank
function, so $R_\sqsubseteq(a) \subset R_\sqsubseteq(b)$. Since $a$ and $b$
      were chosen arbitrarily we know that $R_\sqsubseteq$ must be a strict
      order homomorphism into $\ints$ with the $\subseteq$ ordering. So
      $\sqsubseteq=\subseteq$ as desired. The other three cases follow
      similarly.
\end{description}
\end{proof}

Proposition \ref{EndpointMonotone} applies to the four interval orders,
$\le_W, \ge_W, \subseteq, \supseteq$, as there are only four ways for both
endpoints to be monotone, and each corresponds to one of these four interval
orders. The rank function for the strong interval order $\le_S$ is derived
easily as a case of the weak order $\le_W$.

\begin{cor} \label{corstrong}
Let $R_\sqsubseteq(a) = [r_*(a),r^*(a)]:\po \rightarrow \ints$ be an interval
rank function for interval order $\le_S$ (resp. $\ge_S$). Then, $r_*(a)$ and
$r^*(a)$ are both isotone (resp. antitone).
\end{cor}
\begin{proof}
We have previously observed that if $\overline{x} \le_S \overline{y}$ then
$\overline{x} \le_W \overline{y}$, i.e., the strong interval order implies
the weak interval order. Therefore if $R_\sqsubseteq$ is an interval rank
function for $\le_S$ then it is also an interval rank function for $\le_W$.
So by Proposition \ref{EndpointMonotone}(i) we know that $r_*(a)$ and
$r^*(a)$ are both isotone. And if $\sqsubseteq=\ge_S$ then $r_*(a)$ and
$r^*(a)$ are both antitone (Prop \ref{EndpointMonotone}(ii)).
\end{proof}

Since $\le_S$ implies $\le_W$, but not {\it vice versa}, Corollary \ref{corstrong} is not a
bi-implication, like Proposition \ref{EndpointMonotone}.
\fig{antiStrongCounter} shows a partial order adorned with an interval-valued
function with both endpoints antitone on each of its elements which holds for
$\le_W$, but not $\le_S$.

\begin{figure}[h]
\centering{
  \includegraphics[scale=0.75]{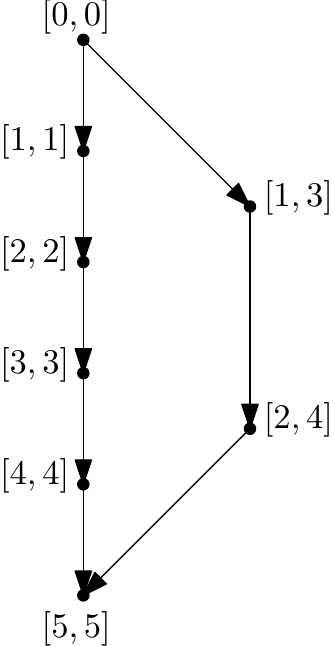}
  \caption{A poset with interval valued strict antitone function which is an
  interval rank function for the weak order, but not for the strong order.}
  \label{antiStrongCounter}
}
\end{figure}

We noted that strict interval orders play a critical role in Def.\
\ref{IntRank} to avoid degeneracies like $R(a)=[r,r]$. But even though the
interval ranks are strictly monotonic, and from Prop. \ref{EndpointMonotone}
their endpoints are monotonic, further degeneracies are possible when even
{\em one} of the two endpoints is not strictly monotonic, i.e.\ for
situations where $a < b, R(a) = [r,r]$, and $R(b) = [r,s]$, while $r < s$.
This is also semantically fraught, as we will usually wish for both endpoints
to be clearly distinguishable when an element changes. We thus have a special
interest in the case where the endpoints are also strictly monotonic.

\begin{defn}[(Strict Interval Rank Function)]
An interval rank function $R_\sqsubseteq \in {\bf R}_\sqsubseteq(\poset)$ is
{\bf strict} if $r_*$ and $r^*$ are strictly monotonic. Let ${\bf
S}_\sqsubseteq(\poset)$ be the set of all strict interval rank functions
$R_\sqsubseteq$ for the interval order $\sqsubseteq$ on a poset $\poset$.
\end{defn}

Requiring strict monotonicity of the endpoints $r_*$ and $r^*$ allows us to
strengthen Proposition \ref{EndpointMonotone} in the case where $\sqsubseteq
\in \{\leq_W, \geq_W, \subseteq, \supseteq\}$.

\begin{prop}\label{EndpointMonotoneStrict}
	Let $F : P \rightarrow \overline{\N}$ be an interval-valued
function on poset $\tup{\poset,\leq}$, so that $F(a) = [F_*(a),F^*(a)] \in
\overline{\N}$. Then, $F_*$ and $F^*$ are strictly monotonic functions iff
$F$ is a strict interval rank function for $\leq_W,\geq_W,\subseteq,$ or
$\supseteq$. In particular:

	\begin{enumerate}[(i)]
  \item $F_*$ and $F^*$ are both strictly isotone iff $\sqsubseteq= \le_W$.
  \item $F_*$ and $F^*$ are both strictly antitone iff $\sqsubseteq=
      \ge_W$.
  \item $F_*$ is strictly antitone and $F^*$ is strictly isotone iff
      $\sqsubseteq = \subseteq$.
  \item $F_*$ is strictly isotone and $F^*$ is strictly antitone iff
      $\sqsubseteq = \supseteq$.
\end{enumerate}
\end{prop}
\begin{proof}\mbox{}
\begin{description}
  \item[$(\Longleftarrow)$] If we assume that $F$ is a strict interval rank
      function then $F_*$ and $F^*$ are strictly monotonic by definition.
      So this direction is trivial.

  \item[$(\Longrightarrow)$] Now, we assume that $F_*$ and $F^*$ are
      strictly monotonic functions on $\poset$ and we must show that $F$ is
      necessarily a strict interval rank function for the appropriate
      $\leq_W,\geq_W,\subseteq,\supseteq$. As in the proof for Proposition
      \ref{EndpointMonotone} we will show only for one of the four cases
      and note that the rest follow similarly. This time we will prove (i).
      Let $F(a) = [F_*(a),F^*(a)]$ be an interval-valued function such that
      $F_*$ and $F^*$ are both strictly isotone on $\poset$. We must show
      that if $a<b \in P$ then $F(a) <_W F(b)$. Because of the strict
      endpoint conditions we know that
      \[ F_*(a) < F_*(b) \text{ and } F^*(a) < F^*(b).\]
      This implies the desired $F(a) <_W F(b)$.
\end{description}
\end{proof}
We note that Proposition \ref{EndpointMonotoneStrict} is stronger than
Proposition \ref{EndpointMonotone} because we do not assume that $F$ is an
interval rank function, only that it is an interval-valued function.

\subsection{Standard Interval Rank}

Within the space of possible strict interval rank functions for $\ge_W$, one
stands out as especially significant.

\begin{defn}[(Standard Interval Rank)]
Let
\[ R^+(a) \define  [ \height(\up a)-1, \height - \height(\down a) ]
\in
\overline{\height-1} \]
be called the {\bf standard interval rank function}. For convenience we also
denote $R^+(a) = \left[ r^t(a), r^b(a) \right]$, where $r^t(a) \define
\height(\up a)-1$ is called the {\bf top rank} and $r^b(a) \define \height -
\height(\down a)$ is called the {\bf bottom rank}.
\end{defn}

$R^+(a)$ is a strict interval rank function for the weak order, although
actually it reverses to $\le_W$, so that its endpoints range naturally and
appropriately from $0$ to $\height-1$ as $a$ ranges from $\top$ to $\bot$.
Moreover, it is also the largest possible strict interval rank function for
the weak order whose image is in $\overline{\height-1}$.

\begin{prop}\label{standard}

For a finite bounded poset $\poset$,

\begin{enumerate}[(i)]
  \item $R^+ \in {\bf S}_{\ge_W}(\poset)$ is a strict interval rank
      function for the weak interval order $\ge_W$;
  \item $R^+(\top) = [0,0];$ $R^+(\bot) = [ \height-1, \height-1]$;
  \item $R^+$ is maximal w.r.t. $\subseteq$ in the sense that $\forall R
      \in {\bf S}_{\ge_W}(\poset)$ with $R(\poset) \subseteq
      \overline{\height-1}$, $\forall a \in P, R(a) \sub R^+(a)$.
      \label{stdIntRankMaxSub}
\end{enumerate}

\end{prop}

\begin{proof}
\mbox{}
\begin{enumerate}[(i)]
\item We first show that $r^t \leq r^b$ so that $R^+$ is an interval-valued
    function. Indeed, $\forall a \in P$ we have that
    \[ \height(\down a) + \height(\up a) \leq \height([0,1])+1 = \height+1 \]
    with equality iff $a \in I(\poset)$ is a spindle element. Rearranging we
    see that $\height(\up a) -1 \leq \height - \height(\down a)$ which is
    precisely $r^t \leq r^b$. Then, it is evident from $\height(\upa)$
    being strictly antitone, $\height(\dna)$ being strictly isotone, and
    negation being order reversing that $r^t,r^b$ are both strictly
    antitone. Therefore, by Prop. \ref{EndpointMonotoneStrict} we have that
    $R^+$ is a strict interval rank function for $\geq_W$.

\item Follows from $\height( \up \top ) = \height( \down \bot ) = 1$ and
    $\height( \down \top ) = \height( \up \bot ) = \height$. \label{one}

\item \label{three} Since any $R \in {\bf S}_{\ge_W}(\poset)$ is a strict
    interval rank function for $\ge_W$ we know that if $a<b \in P$, $R(a)
    >_W R(b)$. Moreover we know
    \[ r_*(a) > r_*(b) \text{ and } r^*(a) > r^*(b) \]
    so that $r_*$ and $r^*$ are strictly antitone. In addition, we are
    restricting to the case where $0 \leq r_*\leq r^* \leq \height-1$.
    Under these assumptions we must show that $\forall a \in P$
    \begin{align*}
    r_*(a) &\geq r^t(a) = \height(\up a)-1\\
    r^*(a) &\leq r^b(a) = \height-\height(\down a).
    \end{align*}

    First notice that, by definition of $\height$, there must be a chain
    $C_{\height} \subseteq P$ of length $\height$ with greatest element
    $\top$ and least element $\bot$. Since $R(\poset) \subseteq
    \overline{\height-1}$ and $r_*, r^*$ are strictly antitone we must have
    $R(\top) = [0,0]$ and $R(\bot) = [\height-1,\height-1]$.

    Now, let $a \in P$, by definition of $\height(\cdot)$ we know that
    there must be a chain $C \subseteq P$ of length $\height(\up a)$ with
    greatest element $\top$ and least element $a$. We already showed that
    $r_*(\top) = 0$. Then, in order for $r_*$ to be strictly antitone we
    need $\forall c_1 < c_2 \in C$, $r_*(c_1) > r_*(c_2)$. Therefore
    $r_*(c)$ must be at least the chain distance from $\top$ to $c$ along
    $C$, less one, for all $c \in C$. In particular, $r_*(a) \geq
    \height(\up a) -1 = r^t(a)$.

    Dually, there must be a chain $D \subseteq P$ of length $\height(\down
    a)$ with greatest element $a$ and least element $\bot$. We already know
    $r^*(\bot)=\height-1$. In order for $r^*$ to be strictly antitone it
    must be true that $\forall d_1 < d_2 \in D$, $r^*(d_1) > r^*(d_2)$.
    Therefore, $r^*(d)$ must be at most $\height-1$ minus the chain
    distance from $\bot$ to $d$ along $D$, less one, for all $d \in D$. In
    particular
    \[ r^*(a) \leq \height-1 - (\height(\down a)-1) = \height-\height(\down a) = r^b(a).\]

\end{enumerate}

\end{proof}

$R^+$ motivates its use in applications by being the largest, most
conservative, strict interval rank function for the natural weak order. But
there are plenty smaller. Consider \fig{counter1}, illustrating the possible
strict interval rank functions for the weak order on the poset known as
$N_5$. With spindle $\bot \cover C \cover A \cover \top$, we must have (see
Prop. \ref{prop1}) $R(\top) = [0,0]$, $R(A) = [1,1]$, $R(C) = [2,2]$,
$R(\bot) = [3,3]$ for all $R \in {\bf R}$.  However, for $B$ there are 3
possible values for a strict interval rank for $\ge_W$. We have $R^+(B) =
[1,2]$ for the standard interval rank, but we could also have $R(B) = [1,1]$
or $[2,2]$, and both of these possibilities are subsets of the standard
interval rank $[1,2]$.

\mytikz{\pFivePossibleRanks}{>=latex, line width=0.75pt}{counter1}{All strict
interval ranks on $N_5$.}

Indeed, in \fig{counter1}, all three possible $R$, including $R^+$, reverse
to the strong interval order $\le_S$, but in general, this is not true
(recall \fig{antiStrongCounter}). It is possible to construct interval ranks
which reverse to $\le_S$, but at the price of being smaller than those that
reverse to $\le_W$, and in particular smaller than $R^+$. For example,
trivial assignments like $\forall a \in P, r(a) = r^t(a)$ or $r(a) = r^b(a)$
do this.

\fig{allexrank2} shows our example from \fig{allex2} but now laid out showing
standard interval rank $R^+(a)$ and with each element centered at the midpoint
$\widehat{R^+(a)}$. Table \ref{ranktable} shows the standard interval rank
quantities details including many of the various other quantities discussed
for our example in \fig{allex2}. The quantities $\tilde{r}_*$, $\tilde{r}^*$,
and $r^f(a)$ will be introduced next.

\mytikz{\IntRankMidpointLayout}{>=latex, line
width=0.75pt}{allexrank2}{Example from \fig{allex2} showing standard interval
rank $R^+(a)$, with each element aligned at the midpoint $\widehat{R^+}(a)$. The
spindle chain is shown in bold, while non-precise elements have a bar attaching
their midpoints to a line indicating the range of the standard interval rank.
}

\begin{table}[h]
\begin{center}
\begin{tabular}{l||c|c|c|c|c|c||c|c|c|c}
      &                  &                      & $\Rpl(a) $     &       &        &                    &              &              &           \\
$a$	  & $\height(\up a)$ & $\height(\down a)$   & $=[ r^t,r^b ]$ & $W(a)$& $S(a)$ & $\widehat{R^+(a)}$ &$\tilde{r}_*(a)$ &$\tilde{r}^*(a)$ & $r^f(a)$  \\ \hline
$\top$&	1                &      5               &   [0,0]        &  0    & 5      & 0.0	               &  0           & 8            & 9         \\
$K$   &	2                &      4               &   [1,1]        &  0    & 5      & 1.0	               &  1           & 7            & 7         \\
$C$   &	2                &      3               &   [1,2]        &  1    & 4      & 1.5	               &  2           & 6            & 6         \\
$B$   &	2                &      2               &   [1,3]        &  2    & 3      & 2.0	               &  3           & 5            & 5         \\
$H$   &	3                &      3               &   [2,2]        &  0    & 5      & 2.0	               &  4           & 4            & 5         \\
$J$   &	3                &      2               &   [2,3]        &  1    & 4      & 2.5	               &  5           & 3            & 4         \\
$E$   &	3                &      2               &   [2,3]        &  1    & 4      & 2.5	               &  6           & 2            & 4         \\
$A$   &	4                &      2               &   [3,3]        &  0    & 5      & 3.0	               &  7           & 1            & 3         \\
$\bot$&	5                &      1               &   [4,4]        &  0    & 5      & 4.0	               &  8           & 0            & 1         
\end{tabular}
\caption{Interval rank statistics for the example in
\fig{allex2}.}
\label{ranktable}
\end{center}
\end{table}

While we have not found our sense of interval-valued rank present in the
literature, it is related to some others. First, we note Freese's definition
\cite{FrR04} of a rank function
	$r^f(a) \define \height + \height(\down a) - \height( \up a)$. Although
$r^f$ is isotone, it is strictly so, so that $-r^f$ is strictly antitone.
Beyond that, the relation to our standard interval rank $R^+$ is
straightforward, if inelegant:
\[ r^t(a) + r^b(a) + r^f(a) = 2 \height - 1. \]

Wild \cite{WiM05} calls $\height([\bot,a])$ the {\bf natural rank} of $a \in
P$, although his posets are pointed to be oriented to $\bot \in P$. And
Schr\"oder \cite{ScB03} introduces a concept we will call {\bf procedural
rank}, since it is based on a recursive algorithm, rather than a closed-form
equation.

\begin{defn}[(Procedural Rank)] \cite{ScB03} \label{posetrank}
Assume a general (possibly unbounded) poset $\poset = \tup{ P, \le }$. Then
for any element $a \in P$, define its procedural rank
recursively as%
\footnote{Schr\"oder's definition is actually defined in the dual for
$\Minnn$, but is translated here as $\Maxxx$ because we point our posets the
opposite way. The results are identical.}
	\({posetrankeq} \tilde{r}_*(a) \define \halffunc{
		0,	& a \in \Maxxx(P)	\\
		k, & a \in \Maxxx \left(
			P \setminus \{ b \in P \st \tilde{r}_*(b) < k \} \right)
	}.	\)
\end{defn}

So procedural rank is determined by recursively ``slicing off'' the maximal
elements of a poset, incrementing the rank counter as we go. This sense of rank
as an element attribute is actually equivalent to our top rank.

\begin{prop} \label{maxrank}
For a finite bounded poset $\poset$, $\tilde{r}_* = r^t$.
\end{prop}

\begin{proof}

Follows directly from Schr\"oder's Prop. 2.4.4 \cite[p.\ 35]{ScB03}, which
states that the (procedural) rank of an element $p \in P$ is the length of
the longest chain in $P$ that has $p$ as its largest element.

\end{proof}

Note that for Wild's natural rank $\height([\bot,a])$, the dual
$\height([a,\top])$ is obviously available as well, again motivating an
interval-valued sense of rank which would involve both. Similarly, we use
$\tilde{r}_*$ for procedural rank suggestively, as its dual function is
readily available as
	\[ \tilde{r}^*(a) \define \halffunc{
		0,	& a \in \Minnn(P)	\\
		k,	& a \in \Minnn \left( P \setminus \{ b \in P \st \tilde{r}^*(b) < k  \}
		\right) }.	\] Where $\tilde{r}_*(a) = 0$ for maximal elements $a \in
\Maxxx(P)$ and is antitone, $\tilde{r}^*(a) = 0$ for minimal elements $a \in
\Minnn(P)$ and is isotone. Thus seeking antitone functions, while avoiding
any one-sided senses of rank, an {\bf interval procedural rank} function can
be easily identified as \[\tilde{R}(a) \define [ \tilde{r}_*(a), \height -
\tilde{r}^*(a) - 1 ].\]

\begin{cor}
$\tilde{R} = R^+$.
\end{cor}
\begin{proof}
Follows from Prop.\ \equ{maxrank} and its dual argument for $\tilde{r}^*$,
and the definition of $r^b$.
\end{proof}

The primary $r^*$ and dual $r_*$ procedural ranks, and Freese's rank
$r^f$, are also shown in Table \ref{ranktable}.

\section{Characterizing Interval Rank}\label{charIntRank}

We now consider a number of the properties and operations of standard
interval rank, including rank precision measured by the width of its standard
interval rank $W( R^+(a) )$; poset spindles as their maximally-long graded
regions; and interval-valued quantitative comparisons of interval ranks in
posets. In this section we heavily refer to our example, in order to
illustrate these attributes.

\subsection{Interval Rank Width, Spindles, and Graded Posets}

We want to describe an element $a \in P$ by the width $W( R^+(a) )$ of its
standard interval rank as a measure of its rank's precision. For this
subsection let
	\[ W(a) \define W( R^+(a) ) = r^t(a) - r^b(a).	\]

To begin with, $\forall a \in P, W(a) + S(a) = \height(\po)$, so for each
element $a \in P$ its interval rank width $W(a)$ and centrality $S(a)$ are
effectively alternate representations of the same concept. In particular, the
width of an element $a$ is minimal at $W(a) = 0$ when its centrality $S(a) =
\height$ is maximal.  Then the largest of the maximal chains of its
hourglass, $\chains(\Xi(a))$, is maximum in $\poset$, so that it contains a
spindle chain; i.e., $\chains(\Xi(a)) \int \spindle(\poset) \neq \emptyset$.
Width $W(a)$ grows to $\height-3$ as centrality $S(a)$ shrinks to $3$ for
$\chains(\Xi(a)) = \{ \bot \cover a \cover \top \}$.

\begin{prop} \label{prop1}
Let $\po=\tup{P,\leq}$ be a bounded poset such that $|P| \geq 2$. For an element
$a \in P$:
\begin{enumerate}[(i)]
  \item $W(a) + S(a) = \height$
  \item $W(a) \in [ 0, \maxx( 0, \height - 3 ) ]$
  \item \label{two} Minimum width elements sit on spindle chains: $W(a) = 0
      \iff a \in I(\poset)$.
  \item Maximum width elements sit on minimum length maximal chains: $W(a) =
      \height-3 \iff \bot \cover a \cover \top \in \chains(\poset)$
\end{enumerate}
\end{prop}
\begin{proof}
\mbox{}
\begin{enumerate}[(i)]
  \item Follows directly from the definitions of $W(a),S(a),r^t(a)$, and
      $r^b(a)$.
      \begin{align*}
      W(a) &= W(\Rpl(a)) = W([\height(\up a)-1, \height-\height(\down a)]) \\
           &= \height-\height(\down a) -(\height(\up a)-1)\\
           &= \height+1-(\height(\down a)+\height(\up a))\\
      S(a) &= \height(\down a)+\height(\up a)-1\\
      W(a) + S(a) &= \height
      \end{align*}
  \item From the definition of $R^+$, in principle $W(a) \le \height - 1$.
      But $W(a) = \height - 1$ would imply that $a=\bot=\top$, which we
      have explicitly discounted by requiring $|P| \ge 2$. The two element
      poset with height 2 has $W(a)=\height-2=0$, and is the only such
      poset with $W(a)=\height-2$. Then, $N_5$ is the simplest poset for
      which $\exists a \in P, W(a) > 0$ (see \fig{counter1}), and has
      elements with $W(a)=\height-3$.
  \item Consider a spindle chain $C \in \chains(\poset)$. Then we know that
      $|C| = \height = \height([\bot,\top])$, so that $C = a_0 \cover a_1
      \cover \ldots \cover a_{\height-2} \cover a_{\height-1}$, where $a_0
      = \bot, a_{\height-1}=\top$. So the range of $r^t$ and $r^b$ being
      $\N_{\height-1}$, together with the strict antitonicity of $r^t$ and
      $r^b$, forces $\Rpl(a_i) = [ i, i ]$. Thus for any $a \in I(\poset),
      W(a) = 0$. Conversely, $W(a) = 0 \implies \height = S(a)$ (by (i)),
      so that $\exists C \in \spindle(\poset), C \in \spindle( \Xi( a ) )$,
      and necessarily $a \in C, a \in I(\poset)$.
  \item By (i), $W(a) = \height - 3$ implies $\height(\down a) +
      \height(\up a) = 4$. It cannot be that $\height(\up a) = 1$ (resp.\
      $\height(\down a) = 1$), since then $a=\top$ (resp.\ $a=\bot$), for
      which we know $W(a)=0$. Thus $\height(\up a) = \height(\down a) = 2$,
      so that $a \cover \top$ and $\bot \cover a$.
\end{enumerate}
\end{proof}

Structures of interest in classical order theory are often JD, and thus
graded \cite{BiG40}. For example, distributive, semi-modular, geometric, and
Boolean lattices are all graded \cite{AiM79}. However, this strong property
is decidedly {\em not} common in ``real world'' posets encountered in the
kinds of ontologies, taxonomies, concept lattices, object-oriented models,
and related databases which we are interested in. In particular, in graded
posets, the entire poset is a spindle, and all elements are precisely ranked
spindle elements.

\begin{prop} \label{graded-prop}
A bounded poset $\poset$ is graded iff $\spindle(\poset) =
\chains(\poset)$ (which is equivalent to $I(\poset) = P$).
\end{prop}
\begin{proof}
If $\poset$ is graded, then it is JD, so that $\forall C,C' \in
\chains(\poset), |C|=|C'|$, and thus $\spindle(\poset) = \chains(\poset)$.
For the converse, if $\forall a \in P, a \in I(\poset)$, then from
Proposition \ref{prop1} part \ref{two}, $W(a)=0$, so that $\rho(a) = r^t(a) =
r^b(a)$ is a rank function.
\end{proof}

\subsection{Quantitative Interval Rank Comparison} \label{quant}

We are interested in considering the quantitative relationships between
interval ranks. We can do so for two elements $a,b \in P$ by subtracting
$R^+(a)$ from $R^+(b)$ using the separation $\|R^+(a),R^+(b)\|$. This gives
an interval-valued measure of the spread of the interval-valued ranks of the
elements, which is further valued by the width of that spread
$W(\|R^+(a),R^+(b)\|)$.

Referring back to Section \ref{intervals}, and identifying
\begin{align*}
\bar{\alpha} &= [ \alpha_*, \alpha^* ] \\
	&:= R^+(b) - R^+(a) \\
	&= [ r^t(b) - r^b(a), r^b(b) - r^t(a) ] \\
	&= [ -\height + \height(\down a) + \height(\up   b) - 1,
		   \height - \height(\up a )  - \height(\down b) + 1 ],
\end{align*}	
then from Table \ref{taborders}, we derive Table \ref{tabstandard} for the
case of standard rank intervals. Table \ref{tabex1} shows the quantitative
relationships for our example from \fig{allex2}.

\begin{table}

\begin{center}

\begin{tabular}{>{$}c<{$}|>{$}c<{$}}

& \bar{\alpha} = [ r^t(b) - r^b(a), r^b(b) - r^t(a) ] \\
\cline{2-2}
R^+(a) \bowtie R^+(b) & \|\bx,\by\| \\
\cline{2-2}
& W(\|\bx,\by\|) \\

\hline \hline \hline

R^+(a) <_S R^+(b) &	
	0 < r^t(b)-r^b(a) < r^b(b)-r^t(a) \\
\cline{2-2} r^t(a) < r^b(a) < r^t(b) < r^b(b) &
	[r^t(b)-r^b(a),r^b(b)-r^t(a)] \\
\cline{2-2}
&	r^b(b)-r^t(a) - r^t(b)+r^b(a) \\

\hline \hline

R^+(a) \circ_\le R^+(b) &
	r^t(b)-r^b(a) < 0 < r^b(a)-r^t(b) < r^b(b)-r^t(a) \\
\cline{2-2} r^t(a) < r^t(b) < r^b(a) < r^b(b) &
	[0, r^b(b)-r^t(a)] \\
\cline{2-2}
&	 r^b(b)-r^t(a)	\\

\hline \hline

R^+(a) \subset R^+(b) &	
	r^t(b)-r^b(a) < 0 < r^b(b)-r^t(a) \\
\cline{2-2} r^t(b) < r^t(a) < r^b(a) < r^b(b) &
	[0,\Maxxx(r^b(a)-r^t(b),r^b(b)-r^t(a))] \\
\cline{2-2}
&	\Maxxx(r^b(a)-r^t(b),r^b(b)-r^t(a)) \\

\end{tabular}

\caption{Relations among quantities of standard rank intervals.}

\label{tabstandard}

\end{center}

\end{table}

\begin{table}

\begin{center}

{\footnotesize

\setlength{\tabcolsep}{2pt}

\begin{tabular}{||>{$}c<{$}|>{$}c<{$}|||*{8}{>{$}c<{$}|>{$}c<{$}||}}
\hline
\hline
&a,R^+(a)&\multicolumn{2}{>{$}c<{$}||}{\top,[0,0]}&\multicolumn{2}{>{$}c<{$}||}{K,[1,1]}&\multicolumn{2}{>{$}c<{$}||}{C,[1,2]}&\multicolumn{2}{>{$}c<{$}||}{B,[1,3]}&\multicolumn{2}{>{$}c<{$}||}{H,[2,2]}&\multicolumn{2}{>{$}c<{$}||}{J,E,[2,3]}&\multicolumn{2}{>{$}c<{$}||}{A,[3,3]}&\multicolumn{2}{>{$}c<{$}||}{\bot,[4,4]}	\\
\hline
\multirow{2}{*}{$a$}&\multirow{2}{*}{$R^+(a)$}&\multicolumn{2}{c||}{$\bowtie$}&\multicolumn{2}{c||}{$\bowtie$}&\multicolumn{2}{c||}{$\bowtie$}&\multicolumn{2}{c||}{$\bowtie$}&\multicolumn{2}{c||}{$\bowtie$}&\multicolumn{2}{c||}{$\bowtie$}&\multicolumn{2}{c||}{$\bowtie$}&\multicolumn{2}{c||}{$\bowtie$}	\\
\cline{3-18}
&&||\cdot,\cdot||&W&||\cdot,\cdot||&W&||\cdot,\cdot||&W&||\cdot,\cdot||&W&||\cdot,\cdot||&W&||\cdot,\cdot||&W&||\cdot,\cdot||&W&||\cdot,\cdot||&W \\
\hline
\hline
\multirow{2}{*}{$\top$}&\multirow{2}{*}{$[0,0]$}&\multicolumn{2}{>{$}c<{$}||}{=}&\multicolumn{2}{>{$}c<{$}||}{\le_S}&\multicolumn{2}{>{$}c<{$}||}{\le_S}&\multicolumn{2}{>{$}c<{$}||}{\le_S}&\multicolumn{2}{>{$}c<{$}||}{\le_S}&\multicolumn{2}{>{$}c<{$}||}{\le_S}&\multicolumn{2}{>{$}c<{$}||}{\le_S}&\multicolumn{2}{>{$}c<{$}||}{\le_S}	\\
\cline{3-18}
&&[0,0]&0&[1,1]&0&[1,2]&1&[1,3]&2&[2,2]&0&[2,3]&1&[3,3]&0&[4,4]&0 \\
\hline
\multirow{2}{*}{$K$}&\multirow{2}{*}{$[1,1]$}&\multicolumn{2}{c||}{}&\multicolumn{2}{>{$}c<{$}||}{=}&\multicolumn{2}{c||}{$\sub$}&\multicolumn{2}{c||}{$\sub$}&\multicolumn{2}{>{$}c<{$}||}{\le_S}&\multicolumn{2}{>{$}c<{$}||}{\le_S}&\multicolumn{2}{>{$}c<{$}||}{\le_S}&\multicolumn{2}{>{$}c<{$}||}{\le_S}	\\
\cline{3-18}
&&&&[0,0]&0&[0,1]&1&[0,2]&2&[1,1]&0&[1,2]&1&[2,2]&0&[3,3]&0	\\
\hline
\multirow{2}{*}{$C$}&\multirow{2}{*}{$[1,2]$}&\multicolumn{2}{c||}{}&\multicolumn{2}{c||}{}&\multicolumn{2}{>{$}c<{$}||}{=}&\multicolumn{2}{c||}{$\sub$}&\multicolumn{2}{c||}{$\supseteq$}&\multicolumn{2}{c||}{$\circ_\le$}&\multicolumn{2}{>{$}c<{$}||}{\le_S}&\multicolumn{2}{>{$}c<{$}||}{\le_S}	\\
\cline{3-18}
&&&&&&[0,0]&0&[0,2]&2&[0,1]&1&[0,2]&2&[1,2]&1&[2,3]&1	\\
\hline
\multirow{2}{*}{$B$}&\multirow{2}{*}{$[1,3]$}&\multicolumn{2}{c||}{}&\multicolumn{2}{c||}{}&\multicolumn{2}{c||}{}&\multicolumn{2}{>{$}c<{$}||}{=}&\multicolumn{2}{c||}{$\supseteq$}&\multicolumn{2}{c||}{$\supseteq$}&\multicolumn{2}{c||}{$\supseteq$}&\multicolumn{2}{>{$}c<{$}||}{\le_S}	\\
\cline{3-18}
&&&&&&&&[0,0]&0&[0,1]&1&[0,2]&2&[0,2]&2&[1,3]&2	\\
\hline
\multirow{2}{*}{$H$}&\multirow{2}{*}{$[2,2]$}&\multicolumn{2}{c||}{}&\multicolumn{2}{c||}{}&\multicolumn{2}{c||}{}&\multicolumn{2}{c||}{}&\multicolumn{2}{>{$}c<{$}||}{=}&\multicolumn{2}{c||}{$\sub$}&\multicolumn{2}{>{$}c<{$}||}{\le_S}&\multicolumn{2}{>{$}c<{$}||}{\le_S}	\\
\cline{3-18}
&&&&&&&&&&[0,0]&0&[0,1]&1&[1,1]&0&[2,2]&0	\\
\hline
\multirow{2}{*}{$J,E$}&\multirow{2}{*}{$[2,3]$}&\multicolumn{2}{c||}{}&\multicolumn{2}{c||}{}&\multicolumn{2}{c||}{}&\multicolumn{2}{c||}{}&\multicolumn{2}{c||}{}&\multicolumn{2}{>{$}c<{$}||}{=}&\multicolumn{2}{c||}{$\supseteq$}&\multicolumn{2}{>{$}c<{$}||}{\le_S}	\\
\cline{3-18}
&&&&&&&&&&&&[0,0]&0&[0,1]&1&[1,2]&1	\\
\hline
\multirow{2}{*}{$A$}&\multirow{2}{*}{$[3,3]$}&\multicolumn{2}{c||}{}&\multicolumn{2}{c||}{}&\multicolumn{2}{c||}{}&\multicolumn{2}{c||}{}&\multicolumn{2}{c||}{}&\multicolumn{2}{c||}{}&\multicolumn{2}{>{$}c<{$}||}{=}&\multicolumn{2}{>{$}c<{$}||}{\le_S}	\\
\cline{3-18}
&&&&&&&&&&&&&&[0,0]&0&[1,2]&1	\\
\hline
\multirow{2}{*}{$\bot$}&\multirow{2}{*}{$[4,4]$}&\multicolumn{2}{c||}{}&\multicolumn{2}{c||}{}&\multicolumn{2}{c||}{}&\multicolumn{2}{c||}{}&\multicolumn{2}{c||}{}&\multicolumn{2}{c||}{}&\multicolumn{2}{c||}{}&\multicolumn{2}{>{$}c<{$}||}{=}	\\
\cline{3-18}
&&&&&&&&&&&&&&&&[0,0]&0	\\
\hline
\hline

\end{tabular}

}

\caption{Quantitative interval comparisons for example.}

\label{tabex1}

\end{center}

\end{table}

\fig{homoone} shows the example of \fig{allexrank2} equipped with standard
interval rank, and with the edges adorned with their separations and widths.
It is instructive to examine Table  \ref{tabex1} and \fig{homoone} for some
interesting observations. Here we abbreviate
$N=\|R^+(a),R^+(b)\|,W=W(\|R^+(a),R^+(b)\|)$.

\myeps{.275}{homoone}{Example poset from \fig{allex2} equipped with interval
ranks $R^+(a)$, with separation and widths on all links.}

\begin{itemize}

\item High separation width $W$ generally results when the base intervals
    $R^+(a),R^+(b)$ are also wide.

\item Naturally all spindle edges have $N=[1,1],W=0$: elements on spindle
    edges have minimal rank interval width, and are on distinct levels, but
    minimally so.

\item Some of these interval rank poset edges have high $W$, for example
    $\tup{C,B}$ and $\tup{B,JE}$. But all have low $N$, in the sense of a
    low midpoint, since they capture the rank differences among all elements
    in $\poset$.

\item Maximum width $W=2$ is additionally attained for pairs $\tup{C,E}$,
    because both have wide interval ranks to begin with.

\item Finally, the pairs $\tup{\top,B},\tup{B,\bot}$ stand out in
    \fig{homoone} with the highest $N=[1,3]$ (midpoint 2) and $W=2$. This
    is indicative of $B$'s position on a minimum size maximum chain, with
    the highest interval rank width and maximum distance from its parent
    and child.

\item But from Table  \ref{tabex1}, there are higher $N$, but involving
    ``long links'' terminating in $\bot$, e.g.\ $\tup{\top,\bot}$.

\end{itemize}

\section{Conclusion and Future Work}


In this paper we have introduced the concept of an interval valued
rank function available on all finite posets, and explored its
properties, including for a canonical standard interval rank. A number
of questions remain unaddressed, some of which we pursue
elsewhere \cite{JoCHoE14a}. These include:

\begin{itemize}

\item In \sec{intervals} above we introduced the concept of conjugate orders which
``partition'' the space of pairwise comparisons of order elements. For
us, these elements are real intervals, and we noted the weak and
subset interval orders are (near) conjugates. Interpreting standard
interval rank in the context of the conjugate subset order $\sub$ is
therefore interesting, as is exploring the possible conjugate orders
to the strong interval order $\le_S$.

\item Interval rank functions establish order homomorphisms from a
base poset to a poset of intervals. It is thus natural to ask about
interval rank functions on that poset of intervals, and thereby a
general iterative strategy for interval ranks.

\item Just as we hold that rank in posets is
    naturally and profitably extended to an interval-valued concept, so
    this work suggests that we should consider extending gradedness from a
    qualitative to a quantitative concept. A graded poset is all spindle,
    with all elements being precisely ranked with width 0, and {\it vice
    versa}. Therefore there should be a concept of posets which fail that
    criteria to a greater or lesser extent, that is, being more or less
    graded. In fact, we have sought such measures of gradedness as
    non-decreasing monotonically with iterations of $R^+$. Candidate
    measures we have considered have included the avarege interval rank
    width, the proportion of the spindle to the whole poset, and various
    distributional properties of the set of the lengths of the maximal
    chains. While our efforts have been so far unsuccessful,
    counterexamples were sometimes very difficult to find, and exploring
    the possibilities has been greatly illuminating.

\end{itemize}



\end{document}